\documentclass[11pt,leqno]{amsart}
\usepackage{amsmath,amssymb,amsthm}
\usepackage{hyperref}

\usepackage{bbm}

\newcommand{\N}{\mathbb{N}}
\newcommand{\R}{\mathbb{R}}

\newcommand{\ep}{\varepsilon}

\renewcommand{\leq}{\leqslant}

\newtheorem{theorem}{Theorem}[section]
\newtheorem{lemma}[theorem]{Lemma}

\newtheorem{proposition}[theorem]{Proposition}
\newtheorem{corollary}[theorem]{Corollary}
\theoremstyle{definition}
\newtheorem{definition}[theorem]{Definition}

\theoremstyle{remark}

\numberwithin{equation}{section}

\def\fnote#1{\footnote}

\def\natu{{\mathbb N}}

\def\real{{\mathbb R}}

\def\K{{\mathbb K}}

\def\ignora#1{}
\def\n3#1{\left\vert  \! \left\vert \! \left\vert \, #1 \, \right\vert \!
  \right\vert \! \right\vert }

\include{draft}

\begin{document}



\title[Weak stable unit ball of $L_1$ preduals]{A characterization of the weak topology in the unit ball of purely atomic  $L_1$ preduals}

\author{Gin\'es L\'opez-P\'erez}\thanks{This research was supported by MICINN (Spain) Grant PGC2018-093794-B-I00 (MCIU, AEI, FEDER, UE) and by Junta de Andaluc\'ia Grant A-FQM-484-UGR18 (FEDER, UE)
and by Junta de Andaluc\'ia Grant FQM-0185.}
\address[G. L\'opez-P\'erez]{Universidad de Granada, Facultad de Ciencias.
Departamento de An\'{a}lisis Matem\'{a}tico, 18071-Granada
(Spain)} \email{ glopezp@ugr.es}
\urladdr{\url{https://wpd.ugr.es/local/glopezp}}

\author{ Rub\'en Medina}\thanks{The second author research has been supported by MIU (Spain) FPU19/04085 Grant.}
\address[R. Medina]{Universidad de Granada, Facultad de Ciencias.
Departamento de An\'{a}lisis Matem\'{a}tico, 18071-Granada
(Spain)} \email{rubenmedina@ugr.es}
\urladdr{\url{https://www.ugr.es/personal/ae3750ed9865e58ab7ad9e11e37f72f4}}

\maketitle 

\begin{abstract}  We study Banach spaces with a weak stable unit ball, that is Banach spaces where every convex combination of relatively weakly open subsets in its unit ball is again a relatively weakly open subset in its unit ball. It is proved that the class of $L_1$ preduals with a weak stable unit ball agree with those $L_1$ preduals which are purely atomic, that is preduals of $\ell_1(\Gamma)$ for some set $\Gamma$, getting in this way a complete geometrical characterization of purely atomic preduals of $L_1$, which answers a setting problem. As a consequence, we prove the equivalence for $L_1$ preduals of different properties previously studied by other authors, in terms of slices around weak stability. Also we get the weak stability of the unit ball of $C_0(K,X)$ whenever $K$ is a Hausdorff and scattered locally compact space and $X$ has a norm stable and weak stable unit ball, which gives the weak stability of the unit ball in $C_0(K,X)$ for finite-dimensional $X$ with a stable unit ball and $K$ as above. Finally we prove that Banach spaces with a weak stable unit ball satisfy a very strong new version of diameter two property.\end{abstract}

\section{Introduction}
A convex subset $C$ of a topological vector space $X$ is said to be stable if the midpoint  map $C\times C\rightarrow C$ given by $(x,y)\to \frac{x+y}{2}$ is open for the restricted topology on $C$. Then $C$ is stable if, and only if,  every convex combination of relatively open subsets in $C$ is again a relatively open subset in $C$ \cite{OB} (do not confuse with the concept of weakly stable Banach spaces in the Krivine-Maurey sense). In the past, the stability has been a relevant tool in different contexts. For example, in \cite{OB} it was proved, in the case $C$ is compact, that the openness of the above midpoint map is equivalent to the openness of the barycentre map (see \cite{Phelps}), which gives the continuity of the convex envelope for every continuous function on $C$ \cite{Ves}. Also, stability implies regularity (see \cite{Clau}), which gives solution to the abstract Dirichlet problem \cite{Clau1}.  Stability is a key to the study of extremal operators, since  for $K$ a compact and Hausdorff space, a norm one element $f\in C(K,X)$, the space of continuous functions on $K$ into a Banach space $X$ with stable unit ball, is an extreme point of the unit ball if, and only if, every value of $f$ is an extreme point in the unit ball of $X$ \cite{Gras}.  More recently, it is proved in \cite{Shiri} the stability of the convex set of positive trace-class operators on a separable Hilbert space with trace one, known as the set of quantum states or density operators, and this fact is the key to study the continuity of Von Neumann entropy, relevant in theoretical quantum information.

The prototype example of a stable compact convex set is a Bauer simplex, that is the positive face of the unit sphere in the dual of a $C(K)$ space or equivalently the set of Radon probability measures for its weak-star topology \cite{OB}. Examples of Banach spaces with a norm stable unit ball are strictly convex spaces or spaces with 3.2 intersection property \cite{Clau}. Also stable convex compact sets in finite-dimensional spaces are completely characterized. Indeed, a compact and convex subset $K$ in $\real^n$ is stable if, and only if, the $p$-skeleton $\{x\in K: {\rm dim (
 face}(x))\leq p\}$ is closed for $0\leq p\leq n-1$ (face(x) denotes the minimal face in $K$ containing $x$, and dim means the dimension of the affine span). Then $K$ is always stable if $n=2$,  $K$ is stable if, and only if, the set of extreme points of $K$ is closed, in the case $n=3$, and for $n>3$, there are compact and convex subsets $K$ with a closed set of extreme points, failing to be stable. As a consequence of the above, every unit ball of a finite-dimensional Banach space with a finite set of extreme points, or such that every point in the unit sphere is an extreme point, is stable \cite{Papa}.

As the above cited example, most of results about stability in the past, have been obtained in the compact case, or concern with norm stability. More recently, motivated by the fact that the positive face of the unit sphere of the space $L_1[0,1]$ is weak stable \cite[Remark IV.5]{ggms}, in relation with other geometrical properties in Banach spaces as diameter two properties  \cite{lo} (also octahedrality, Daugavet property, etc \cite{jfa}), the weak stability has attracted the attention of researchers in geometry of Banach spaces, and new examples of weak stable unit balls of Banach spaces have emerged. For example, in \cite{abralim} and \cite{ha} it is proved that the unit ball of $C_0(K)$, for a Hausdorff and locally compact space $K$, is weak stable if, and only if, $K$ is scattered (see also \cite{abralim2} and \cite{be}). As it can be seen in the above references, the current interest is to know how big is the class of Banach spaces with a weak stable unit ball, and the natural questions in the setting of the above papers is to know what preduals of $L_1$ or $C_0(K,X)$ spaces have a weak stable unit ball.

From now, we will say that a Banach space $X$ has a weak stable unit ball $B_X$, if it is stable for the inherited  weak topology (other names have been used to mention stability in recent papers, but it seems natural keeping the used notation in the past). In the case $B_X$ is stable for the norm topology, we will say that $X$ has a stable unit ball.

Observe that a Banach space with a weak stable unit ball satisfies the strong diameter two property \cite{lo}. Then $X^*$ is octahedral from \cite{jfa} and so the class of $L_1$ preduals is a natural class to looking for weak stable unit balls. Our main result in Section 2 proves that a $L_1$ predual $X$ is weak stable if, and only if, $X^*$ is purely atomic, that is, $X^*=\ell_1(\Gamma)$ for some set $\Gamma$, which gives a complete geometric characterization of purely atomic $L_1$ preduals. Taking into account the existence of preduals of $\ell_1$  which are not isomorphic to complemented subspaces of $C(K)$ spaces \cite{ben}, the above fact improves the aforementioned results around weak stability, known up to now. Our main result uses continuous selections of multivalued maps techniques in the setting of $L_1$ preduals. In Section 3, after introduce some easy consequences of Banach spaces with a stable unit ball, we prove using again multivalued maps techniques the weak stability of the unit ball of $C_0(K,X)$, whenever $K$ is a Hausdorff and scattered locally compact space and $X$ has a stable and weak stable unit ball, which improves the known results up to now. Observe for example that in \cite{abralim2}, the above result is obtained for finite-dimensional $X$ satisfying an hypothesis which we will prove to be equivalent to norm stability of the unit ball in $X$. We finish the paper by showing that Banach spaces with a weak stable unit ball satisfy a stronger new version among the known diameter two properties.
 
 We use standard terminology for Banach spaces. $S_X$ and $B_X$ denote the unit sphere and the closed unit ball, respectively, of a Banach space $X$. $w$ and $w^*$ denote the weak and weak-star topologies and $X^*$ is the topological dual of $X$. A slice in $B_X$ is the intersection of $B_X$ with an open semi-space. A face of a subset $C$ in a Banach space $X$ is a subset $F\subset C$ satisfying that $x,y\in F$ whenever $x,y\in C$ and $tx+(1-t)y\in F$ for some $t\in [0,1]$. An extreme point of $C$ is a point  $x\in C$ such that the singleton $\{x\}$ is a face of $C$. ${\rm co}(A)$ denotes the convex hull of $A$. A $L_1$ predual is a Banach space $X$ so that $X^*$ is linearly isometric to $L_1(\mu)$ for some measure space $(\Omega, \Sigma, \mu)$. If $\Gamma$ is a set, $\ell_1(\Gamma, \mathbb{K})$ is the classical Banach space of (absolutely) summable families of scalars in $\mathbb{K}=\real$ or $\mathbb{C}$.

Now we pass to introduce some notation and known results, which will be used in the following. For $X$ and $Y$ topological spaces, we denote by ${\mathcal P}(Y)$ the power set of $Y$. A map defined on $X$ with values into ${\mathcal P}(Y)$ will be called a multivalued map $\phi:X\rightarrow Y$ so that $\phi(x)$ is a nonempty subset of $Y$ for every $x\in X$. We recall that $\phi$ is called lower semicontinuous (l.s.c) or hemicontinuous at $x\in X$ if for every open subset $U$ of $Y$ such that  $\phi(x)\cap U\neq\emptyset$, there is an open subset $V$ of $X$ with $x\in V$ such that $\phi(z)\cap U\neq\emptyset$ whenever $z\in V$. $\phi$ is called lower semicontinuous if $\phi$ is lower semicontinuous at $x$, for every $x\in X$. Also, $\phi$ is lower semicontinuous if, and only if, the set $\phi^{\ell}(U):=\{x\in X:\phi(x)\cap U\neq \emptyset \}$ is an open subset of $X$ for every open subset $U$ of $Y$. Then a map $f:X\rightarrow Y$ between topological spaces $X$ and $Y$ is open if, and only if, the multivalued map $f^{-1}:Y\rightarrow X$ given by $f^{-1}(y)=\{x\in X:f(x)=y\}$ is l.s.c. A multivalued map $\phi:X\rightarrow Y$, with $X$ and $Y$  convex subsets of real or complex vector spaces, is called convex if $ t\phi(x)+(1-t)\phi(y) \subset \phi(tx+(1-t)y)$ for every $t\in [0,1]$ and $x,y\in X$. $\phi$ is called symmetric if $\phi(zx)=z\phi(x)$ for every $x\in X$ and $z$ a modulus one complex scalar. A continuous selection of a multivalued map $\phi:X\rightarrow Y$ will be a continuous map $f:X\rightarrow Y$ such that $f(x)\in \phi(x)$ for every $x\in X$. The well known Michael selection principle asserts that the above $\phi$ has a continuous selection whenever $X$ is a paracompact topological space, $Y$ is a Fr\'echet space, $\phi$ is l.s.c and $\phi(x)$ is a closed and convex subset of $Y$ for every $x\in X$. Finally, we recall a useful criteria for l.s.c. multivalued maps (see \cite[Chapter 17]{Ali} for background about the above topics).
We will use the above freely in the sequel.

\begin{lemma}\label{redes} (\cite[Theorem 17.21]{Ali}) For $X$ and $Y$ topological spaces, and $\phi:X\rightarrow Y$ a multivalued map, the following assertions are equivalent:\begin{enumerate}\item[i)] $\phi$ is l.s.c.\item[ii)] For every net $\{x_{\alpha}\}$ in $X$ converging to $x\in X$, and for every $y\in \phi(x)$, there is a subnet $\{x_{\alpha_{\lambda}}\}$ of $\{x_{\alpha}\}$ and a net $\{y_{\lambda}\}$ in $Y$ such that $y_{\lambda}\in\phi(x_{\alpha_{\lambda}})$ for every $\lambda$ and $\{y_{\lambda}\}$ converges to $y$.\end{enumerate} Moreover, net and subnet can be changed in ii) by sequence and subsequence, whenever $X$ and $Y$ are first countable.\end{lemma}

\section{Purely atomic $L_1$-preduals}

As we said in the introduction, our goal is to enlarge the known class of Banach spaces with a weak stable unit ball. It is known that $C(K)$ spaces, with $K$ a scattered, Hausdorff and compact topological space, are weak stable \cite{abralim}, \cite{ha}. Also, other properties around weak stability has been considered by different authors (see \cite{abralim, abralim2, be, ha, lo}), like the weak relative openess of convex combination of slices in the unit ball, nonempty relative weak interior of convex combinations of slices  in the unit ball, nonempty relative  weak interior of convex combinations  of relative weak open subsets in the unit ball, or property ($\overline{P1}$). A Banach space $X$ is said to verify the property $\overline{\rm P1}$  if for every convex combination of slices in $B_X$, $C$, and for every $x\in C$ there is a weak open subset relative to $B_X$, $U$, so that $x\in U$ and $U\subset \overline{C}$. Clearly weak stability implies all the above properties. With the above in mind, it is natural trying to characterize the weak stability for the class of preduals of $L_1$, improving the above known results. This will be our first goal, and for that, our main tool will be the next well known result about continuous selections in the setting of real or complex $L_1$ preduals.
\begin{theorem}\label{lazar} (\cite[Theorem 2.2]{La1}, \cite{edu}) Let $X$ be a real or complex Banach space with $X^*$ isometric to $L_1(\mu)$ space and $E$ be a Fr\'echet space. Consider a convex, symmetric and $w^*-$l.s.c. multivalued map $\phi:B_{X^*}\rightarrow E$ such that $\phi(x^*)$ is a convex and closed subset of $E$ for every $x^*\in B_{X^*}$. Then $\phi$ admits an affine, symmetric and $w^*-$continuous selection $h:B_{X^*}\rightarrow E$. Moreover, if $F$ is a face of $B_{X^*}$ with $H:={\rm co}(F\cup -F)$ $w^*-$closed and $g:H\rightarrow E$ is an affine, symmetric and $w^*-$continuous selection of $\phi_{\mid H}$, then $h$ can be chosen so that $h_{\mid H}=g$.\end{theorem}

The above result was proved in \cite[Theorem 2.2]{La1} in the real case. The complex case appears in \cite{edu}, assumming that $H$ is the intersection of $B_{X^*}$ with a $L$-ideal. This intersection agree with the concept of biface in \cite{La1} in the case of real preduals of $L_1$, as said in \cite{edu} (see also \cite[Pag. 168]{Alfsen}).  

Before showing our main result we need a simple lemma.



\begin{lemma}\label{simplex} Let $\Gamma$ a nonempty set and $E:=\{e_{\gamma}:\gamma\in \Gamma\}$ the set of basis vectors in $\ell_1(\Gamma, \mathbb{K})$. Then the linear span of $E$ is norm dense in $\ell_1(\Gamma, \mathbb{K})$,  and for $C$ a finite subset of $E$, one has that ${\rm co}(C)$ is a $w^*-$closed face of $B_{\ell_1(\Gamma, \mathbb{K})}$ and ${\rm co}(C)$ is a simplex, that is, every element in ${\rm co}(C)$ has a unique expression as a convex combination of elements in $C$.\end{lemma}

\begin{proof} Lets take $G=\{\gamma_1,\dots,\gamma_k\}\subset \Gamma$ and define $C=\{e_\gamma,\;\gamma\in G\}\subset E$. In order to show that ${\rm co}(C)$ is a face of $B_{\ell_1(\Gamma, \mathbb{K})}$, let $\lambda_1,\dots,\lambda_k\in(0,1)$ be such that $\sum\limits_{i=1}^k\lambda_i=1$ and take the point $x=\sum\limits_{i=1}^k\lambda_ie_{\gamma_i}$. Given $a,b\in B_{\ell_1(\Gamma,\mathbb{K})}$ such that $\frac{a+b}{2}=x$ then we state that $a,b\in co(C)$:

First of all, if we fix $\gamma\in \Gamma\setminus G$, then $b(\gamma)=-a(\gamma)$ so we have
$$\begin{aligned}2\ge||a||+||b||&\ge2|a(\gamma)|+\sum\limits_{i=1}^k|a(\gamma_i)|+|b(\gamma_i)|\ge 2|a(\gamma)|+\sum\limits_{i=1}^k|a(\gamma_i)+b(\gamma_i)|\\&=2|a(\gamma)|+2\sum\limits_{i=1}^k\lambda_i=2|a(\gamma)|+2,\end{aligned}$$
which arises that $a(\gamma)=b(\gamma)=0$ for every $\gamma\in \Gamma\setminus G$. The real case is now done. Lets consider  $\mathbb{K}=\mathbb{C}$ for the rest of the proof.
If we suppose that there exists $j\in\{1,\dots,k\}$ such that $a(\gamma_j)\notin\R$, then 
$$\begin{aligned}||a||+||b||&=\sum\limits_{i=1}^k|a(\gamma_i)|+|b(\gamma_i)|>\sum\limits_{i=1}^k|Re(a(\gamma_i))|+|Re(b(\gamma_i))|\\&\ge\sum\limits_{i=1}^k|Re(a(\gamma_i)+b(\gamma_i))|=2,\end{aligned}$$
which is a contradiction with the fact that $a,b\in B_{\ell_1(\Gamma,\mathbb{K})}$.\\
The rest of the proof is straightforward.
\end{proof}

Now we pass to get the characterization of weak stable $L_1$ predual unit balls.

\begin{theorem}\label{predual} Let $X$ be a real or complex isometric predual of a $L_1(\mu)$ space for some measure $\mu$. Then $B_X$ is weak stable if, and only if, $\mu$ is purely atomic.\end{theorem}

\begin{proof} If $X$ is a real Banach space with a weak stable unit ball, in particular every convex combination of slices in $B_X$ has nonempty relative weak interior in $B_X$. Then  we deduce that $\mu$ is purely atomic, from \cite[Theorem 4.7]{lo} (note that  \cite[Theorem 4.7]{lo} is deduced from \cite[Theorem 4.1]{lo} and a localizable measure space is used, however $L_1(\mu)$ is linearly isometric to a $\ell_1-$sum of spaces $L_1(\mu_i)$ with $\mu_i$ localizable from \cite[Pag. 136]{Lacey}, see also \cite[Pag. 501]{Defant}). The above is valid in the real case, but it can be easily adapted too for the complex case. 


Assume now that $\mu$ is purely atomic, so $X^*=\ell_1(\Gamma)$ for some set $\Gamma$. Pick $O_1$, $O_2$ relatively weak open subsets of $B_X$ and $x\in\frac{O_1+O_2}{2}$. Put $x=\frac{z_1+z_2}{2}$, with $z_i\in O_i$, $i=1,2$. Then there is $1>\rho>0$ such that $s_i:=(1-\rho) z_i+\rho x
\in O_i$, $i=1,2$. Now \begin{equation}\label{rho}
\vert x^*(s_i)\vert\leq 1-\rho+\rho\vert x^*(x)\vert\ \forall x^*\in B_{X^*}.
\end{equation}
As the subspace generated by the set $E$ of basis vectors of $B_{X^*}$ is dense in $X^*$ for the norm topology from Lemma \ref{simplex}, we can choose  weak neighborhoods $U_i$ of $s_i$, relative to $B_X$, satisfying $U_i\subset O_i$, given by $$U_i=\{z\in B_X: |x_{ij}^*(z-s_i)|<\varepsilon,\ j\in{1,\dots,k_i}\},$$
where $\varepsilon>0$, $k_1,k_2\in \natu$ and $x_{ij}^*\in E$ for $i=1,2$ and $1\leq j\leq k_i$. Now, our goal will be to prove that $\frac{U_1+U_2}{2}$ is a weak neighborhood of $x$ relative to $B_X$. In order to do that, define $C:=\{x_{ij}^*\}_{\substack{i\in\{1,2\}\\j\in\{1,\dots,k_i\}}}$ and $\delta:=\min\limits_{\substack{x^*\in C\\|x^*(x)|<1}}\{1-|x^*(x)|\}>0$, in the case  $A_x^C:=\{x^*\in C: |x^*(x)|<1\}\neq \emptyset$, otherwise we put $\delta:=1$ (observe that $\vert x^*(x)\vert \leq 1-\delta$ whenever $x^*\in A_x^C$). Also we consider $$U:=\{z\in B_X\:|x^*(z-x)|<\mu\ \forall x^*\in C\},$$ for  $0<\mu<\min\{\rho\delta,\varepsilon\}$. As $U$ is a neighborhood of $x$ relative to $B_X$, it is enough to show that $U\subset \frac{U_1+U_2}{2}$. In order to prove that, fix $y\in U$ and put $F={\rm co}(C)$. As $C$ is a finite subset  of points in $E$, we have that $F$ is a $w^*-$closed face of $B_{X^*}$ from Lemma \ref{simplex}, and so $H:={\rm co}(S_{\mathbb{K}}F)$ is $w^*-$closed (in the complex case, even in the real one,  we have that $H$ agree with the intersection of an $L$- ideal, the linear span of $C$, with $B_{X^*})$. In order to apply Theorem \ref{lazar}, we define $g_1,g_2:C\rightarrow \K$ by 
$$g_i(x^*)=\begin{cases}x^*(s_i+y-x)\;\;\;&\text{ if }\; x^*\in A_x^C\ \text{and}\ \text{so}\ |x^*(x)|\le1-\delta\\
x^*(y)&\text{ if }\;|x^*(x)|=1\end{cases}$$
It is clear that $g_1$ and $g_2$ admit unique affine and symmetric extensions to $H$ from Lemma \ref{simplex}, which we will call again $g_1,g_2:H\rightarrow \mathbb{K}$. In order to prove that $g_i(x^*)\in B_{\mathbb{K}}$ for $i=1,2$, $x^*\in H$, it is enough to see that $g_i(x^*)\in B_{\mathbb{K}}$ for every $x^*\in C$. Indeed, in the case $x^*\in C$ with $|x^*(x)|\le 1-\delta$ we have, applying \eqref{rho} and taking into account that $y\in U$ with $\mu<\rho\delta$, $$\begin{aligned}|g_i(x^*)|=&|x^*(s_i+y-x)|\le |x^*(s_i)|+|x^*(y-x)|\\\le& 1-\rho+\rho|x^*(x)|+\mu<1-\rho(1-|x^*(x)|-\delta)\le1,\end{aligned}$$
while in the case $x^*\in C$ with $|x^*(x)|=1$, we have $|g_i(x^*)|=|x^*(y)|\le 1$.

Now we define the multivalued map $\phi:B_{X^*}\rightarrow \mathbb{K}\times\mathbb{K}$ by $$\phi(x^*)=\left\{(r_1,r_2)\in B_{\mathbb{K}}\times B_\mathbb{K}:\frac{r_1+r_2}{2}=x^*(y)\right\}\ \forall x^*\in B_{X^*}.$$
It is easy to see that $\phi$ is an convex and symmetric multivalued map, with $\phi(x^*)$ a closed and convex subset of $\mathbb{K}\times\mathbb{K}$, for every $x^*\in B_{X^*}$. The fact that $\phi$ is a $w^*$-l.s.c. multivalued map is consequence that $\phi=f^{-1}\circ \hat{y}$, where $f:B_{\mathbb{K}}\times B_{\mathbb{K}}\rightarrow B_{\mathbb{K}}$ is the midpoint map, and $\hat{y}$ is the natural injection of $y$ in $X^{**}$ (observe that $\hat{y}$ is $w^*$-continuous as an element of $X^{**}$ and $f^{-1}$ is a l.s.c multivalued map since $B_{\mathbb{K}}$ is stable, that is since $f$ is open).

We will check now that $g:=(g_1,g_2)$ is a symmetric, affine and $w^*$-continuous selection map of $\phi_{\mid H}$. It is clear that $g$ is symmetric, affine and $w^*$-continuous. Since $\phi$ is symmetric and convex, it is enough to show that $g(x^*)\in \phi(x^*)$ for every  $x^*\in C$ in order to prove that $g(x^*)\in \phi(x^*)$ for every $x^*\in H$. In fact, in the case $x^*\in C$ with $|x^*(x)|\le 1-\delta$ we have $$\begin{aligned}\frac{g_1(x^*)+g_2(x^*)}{2}=&\frac{x^*(s_1+y-x+s_2+y-x)}{2}\\=&x^*\left(\frac{s_1+s_2}{2}\right)+x^*(y-x)=x^*(x)+x^*(y-x)=x^*(y),\end{aligned}$$ while in the case $x^*\in C$ with $|x^*(x)|=1$ we have $$\frac{g_1(x^*)+g_2(x^*)}{2}=\frac{x^*(y)+x^*(y)}{2}=x^*(y).$$
Then we have proved that $g$ is a symmetric, affine and $w^*$-continuous selection map of $\phi_{\mid H}$. We apply now Theorem \ref{lazar} to find out  $h:=(h_1,h_2)$, a symmetric, affine and $w^*$-continuous selection map of $\phi$, satisfying that $h_{\mid H}=g$. Define, for $i=1,2$ $$x_i:X^*\rightarrow \mathbb{K}$$
by
$$x_i(x^*)=\begin{cases}0\;\;&\text{ if }\;\;x^*=0\\
||x^*||h_i\left(\frac{x^*}{||x^*||}\right)\;\;&\text{ if }\;\;x^*\in X^*\setminus\{0\}\end{cases}$$
From Krein-Smulian Theorem, we can assume that $x_1, x_2$ are $w^*$-continuous linear functionals on $X^*$ such that $x_i(B_{X^*})\subset B_{\mathbb{K}}$. So we assume that $x_1, x_2\in B_X$ and for $0\neq x^*$, as $(h_1,h_2)$ is a selection of $\phi$, we have that
$$\begin{aligned}x^*\left(\frac{x_1+x_2}{2}\right)=&\frac{x^*(x_1)+x^*(x_2)}{2}=\frac{||x^*||h_1\left(\frac{x^*}{||x^*||}\right)+||x^*||h_2\left(\frac{x^*}{||x^*||}\right)}{2}\\=&||x^*||\left(\frac{h_1+h_2}{2}\right)\left(\frac{x^*}{||x^*||}\right)=||x^*||\left(\frac{x^*}{||x^*||}\right)(y)=x^*(y).\end{aligned}$$
So we deduce that $\frac{x_1+x_2}{2}=y$. 

Finally, we show that $x_i\in U_i$ for $i\in\{1,2\}$. In the case $x^*\in C$ with $|x^*(x)|\le1-\delta$, one has
$$\begin{aligned}|x^*(x_i-s_i)|=&\left|||x^*||h_i\left(\frac{x^*}{||x^*||}\right)-x^*(s_i)\right|=|g_i(x^*)-x^*(s_i)|\\=&|x^*(s_i+y-x)-x^*(s_i)|=|x^*(y-x)|<\mu<\varepsilon,\end{aligned}$$
while in the case $x^*\in C$ with $|x^*(x)|=1$ we deduce from $\frac{x^*(s_1)+x^*(s_2)}{2}=x^*(x)$ and $\vert x^*(s_i)\vert\leq 1$, that $x^*(x)=x^*(s_1)=x^*(s_2)$. Then $$\begin{aligned}|x^*(x_i-s_i)|=&\left|||x^*||h_i\left(\frac{x^*}{||x^*||}\right)-x^*(s_i)\right|=|g_i(x^*)-x^*(s_i)|\\=&|x^*(y)-x^*(s_i)|= |x^*(y-x)|<\mu<\varepsilon.\end{aligned}$$ So $x_i\in U_i$ for $i\in\{1,2\}$ with $\frac{x_1+x_2}{2}=y$ and we are done \end{proof}

It is worth to mention that inside the class of isometric preduals of $\ell_1$, there are Banach spaces not isomorphic to complemented subspaces of any $C(K)$ space \cite[Theorem 2.13]{ben}, so that Theorem \ref{predual} definitively enlarges the class of spaces with a weak stable unit ball far away, even isomorphically speaking, to $C(K)$ spaces known up to now, in a natural way. Also, as a consequence of the above result, we get the equivalence between different considered properties up to now by different authors, around weak stability, in the natural setting of $L_1$ preduals.

\begin{corollary}\label{equivalence} Let $X$ be an isometric predual of some $L_1(\mu)$ space. Then the following assertions are equivalent:\begin{enumerate} \item[i)] $B_X$ is weak stable.\item[ii)] Every convex combination of slices in $B_X$ is a relatively weakly open subset in $B_X$.\item[iii)] Every convex combination of slices in $B_X$ has nonempty relatively weak interior  in $B_X$. \item[iv)] Every convex combination of relatively weakly open subsets in $B_X$  has nonempty relatively weak interior in $B_X$. \item[v)] $X$ satisfies the property $\overline{\rm P1}$.\end{enumerate}Moreover, one of the above assertions holds if, and only if, $\mu$ is purely atomic, that is $X^*$ is isometrically isomorphic to $\ell_1(\Gamma)$ for some set $\Gamma$.\end{corollary}

\begin{proof} i) $\Rightarrow$ ii) $\Rightarrow$ iii), i) $\Rightarrow$ iv) $\Rightarrow$ iii) and i) $\Rightarrow$ iv) are clear. iii) $\Rightarrow$ i) is a consequence of \cite[Theorem 4.7]{lo} and Theorem \ref{predual}. v) $\Rightarrow$ i) is a consequence of \cite[Theorem 2.13]{be}, \cite{hay} and Theorem \ref{predual}. \end{proof}

We don't know the exact relation between the assertions i) to v) in Corollary \ref{equivalence} for general Banach spaces. In particular, we don't know if assertions i) and ii) in Corollary \ref{equivalence} are equivalent for general Banach spaces.  

Thanks to Theorem \ref{predual}, we can partially answer a question posed in \cite{abralim2}, wether the injective tensor product of 2 Banach spaces with weak stable unit ball has a weak stable unit ball.\\

Given $\ell_1(\Gamma_1)$ and $\ell_1(\Gamma_2)$, we are able to completely identify $\ell_1(\Gamma_1,\ell_1(\Gamma_2))$ and $\ell_1(\Gamma_1\times\Gamma_2)$ by the identification
$$T:\ell_1(\Gamma_1,\ell_1(\Gamma_2))\rightarrow\ell_1(\Gamma_1\times\Gamma_2)$$
where $T(x)(\gamma_1,\gamma_2)=x(\gamma_1)(\gamma_2)$ for every $\gamma_i\in \Gamma_i$ with $i=1,2$. It is well known that $\ell_1(\Gamma_1,\ell_1(\Gamma_2))=\ell_1(\Gamma_1)\widehat{\otimes}_\pi\ell_1(\Gamma_2)$ so the next result arises as a direct consequence of this and Theorem \ref{predual}.

\begin{corollary}
Let $X$ and $Y$ be $L_1$ preduals, then the injective tensor product of $X$ and $Y$ has a weak stable unit ball if and only if both $X$ and $Y$ have weak stable unit balls.
\end{corollary}

Let us now get into another special case of injective tensor product, the one in which one of the factors is a $C(K)$ space.

\section{Weak stability in the unit ball of $C(K,X)$}

First, we are going to remind the concept of stable unit ball for the norm topology.

\begin{definition}
Let $X$ be a Banach space. We say that $B_X$ is stable if it is stable for the norm topology, that is, if every convex combination of relatively norm open subsets of $B_X$ is relatively open in $B_X$ for the norm topology.
\end{definition}

Stability and weak stability seem to have a very strong connection. In fact, some of the main known properties for weak stable unit balls (see \cite{abralim} and \cite{ha}) are also true for (norm) stable unit balls, as the following results show.

\begin{proposition}\label{estabilidades}
Let $X$ and $Y$ be Banach spaces with stable unit balls, then:
\begin{enumerate}
\item[i)] If $Z$ is a 1-complemented subspace of $X$, $B_Z$ is stable.\label{complementados}
\item[ii)] If $W=X\oplus_\infty Y$ then $B_W$ is stable.\label{infty}
\end{enumerate}
\begin{proof}

i) Let $P:X\rightarrow Z$ be the norm 1 projection. If $X$ is stable, let $O_1,O_2\subset B_Z$ relatively norm open subsets, and $\lambda\in(0,1/2]$. We call $O_Z=(1-\lambda)O_1+\lambda O_2$, so that $O_X=(1-\lambda)(P^{-1}(O_1)\cap B_X)+\lambda(P^{-1}(O_2)\cap B_X)$ satisfy $O_Z=O_X\cap B_Z$ and i) is proved.

ii) Let us take $w_i=(x_i,y_i)\in B_W$ for $i=1,2$ and consider $\lambda\in(0,1/2]$ and $\delta>0$. Considering now $B_\lambda=(1-\lambda)B_1+\lambda B_2$ where $B_i=B_W(w_i,\delta)\cap B_W$ for $i=1,2$ we prove ii) since
$$B_\lambda=((1-\lambda)B_1^X+\lambda B_2^X)\times((1-\lambda)B_1^Y+\lambda B_2^Y)$$
where $B_i^X=B_X(x_i.\delta)$ y $B_i^Y=B_Y(y_i,\delta)$ for $i=1,2$.
\end{proof}
\end{proposition}

\begin{proposition}\label{C0suma}
If $(X_n)$ is a sequence of Banach spaces with stable unit ball, then $X=C_0(X_n)$ has a stable unit ball.
\begin{proof}
Let $x,s^1,s^2\in C_0(X_n)$ and $\lambda>0$ be such that $(1-\lambda)s^1+\lambda s^2=x$ where $x=(x_n)$, $s^1=(s^1_n)$ and $s^2=(s^2_n)$. We define for arbitrary $\ep\in(0,1)$ the set
$$B_x=(1-\lambda)B_{C_0(X_n)}(s^1,\ep)\cap B_{C_0(X_n)}+\lambda B_{C_0(X_n)}(s^2,\ep)\cap B_{C_0(X_n)}.$$
Let us now take $N\in \N$ such that $\forall n> N$, $||s^i_n||<1-\ep$ for $i=1,2$. For every $n\le N$, we first consider $\delta_n>0$ satisfying
$$B_{X_n}(x_n,\delta_n)\cap B_{X_n}\subset (1-\lambda)B_{X_n}(s^1_n,\ep)\cap B_{X_n}+\lambda B_{X_n}(s^2_n,\ep)\cap B_{X_n}:=B_{x_n}.$$
Now, taking $0<\delta<\min\limits_{n=1,\dots,N}\{\delta_n,\ep\}$, we claim that $B_{C_0(X_n)}(x,\delta)\cap B_{C_0(X_n)}\subset B_x$:\\

If $y=(y_n)\in B_{C_0(X_n)}(x,\delta)\cap B_{C_0(X_n)}$, for $n\le N$, then $y_n\in B_{X_n}(x_n,\delta)\cap B_{X_n}\subset B_{x_n}$. Otherwise, if $n>N$ then $y_n=(1-\lambda)y_n^1+\lambda y_n^2$ where
$$y_n^1=s^1_n+y_n-x_n\in B_{X_n}(s^1_n,\ep)\cap B_{X_n} \;\;,\;\; y^2_n=s^2_n+y_n-x_n\in B_{X_n}(s^2_n,\ep)\cap B_{X_n}$$
and we are done.
\end{proof}
\end{proposition}

It can be easily shown that there exist Banach spaces with stable but not weak stable unit balls such as strictly convex spaces \cite{Clau}. However, all the examples that we have of weak stable unit balls of Banach spaces are also norm stable. We don't know if weak stability implies norm stability.\\

As we said in the introduction, our main goal is to get conditions to have (weak) stability in $C_0(K,X)$. For the weak stability, we will use strongly multivalued maps techniques. So, we begin by showing an easy consequence of these techniques, which it will be useful to get our main result.

\begin{lemma}\label{sequences}  Let $C$ be a stable convex subset of a metrizable topological vector space $X$. Consider a sequence $\{x_n\}$ in $C$ and $x\in C$ such that  $\{x_n\}$ converges to $x$.  If $p_1,p_2, q_1,q_2\in C$ satisfying that $x=\frac{p_1+p_2}{2}=\frac{q_1+q_2}{2}$, then there are $\{p_1^n\}$, $\{p_2^n\}$, $\{q_1^n\}$ and $\{q_2^n\}$ sequences in $C$, and there is $\{x_{\sigma(n)}\}$ a subsequence of $\{x_n\}$, such that $x_{\sigma(n)}=\frac{p_1^n+p_2^n}{2}=\frac{q_1^n+q_2^n}{2}$ for every $n\in\natu$, with $\{p_i^n\}$ converging to $p_i$ and $\{q_i^n\}$ converging to $q_i$, for $i=1,2$.\end{lemma}

\begin{proof} Consider $\phi:C\times C\rightarrow C$ the midpoint map on $C$, given by $\phi(x,y)=\frac{x+y}{2}$. As $C$ is stable, we have that $\phi$ is open, so the map $\phi\times\phi:(C\times C)\times (C\times C)\rightarrow (C\times C)$ given by $(\phi\times\phi)(x,y,z,w)=(\phi(x,y),\phi(z,w))$ is also open. Then the multivalued map $(\phi\times\phi)^{-1}$ is l.s.c. Since $\{x_n,x_n\}$ converges to $(x,x)$, we get, from Lemma \ref{redes}, the desired conclusion.\end{proof}

We obtain now our result about stability of $C_0(K,X)$.

\begin{theorem}\label{vectorvalued} Let $X$ be a stable and weak stable Banach space and $K$ a scattered Hausdorff topological space. Then $C_0(K,X)$ is weak stable.\end{theorem}

\begin{proof} Consider $O_1, O_2$ nonempty relative weakly open subsets of $B_{C_0(K,X)}$. We are going to prove that $\frac{O_1+O_2}{2}$ is  again a relative weakly open subset of $B_{C_0(K,X)}$. Pick $x\in \frac{O_1+O_2}{2}$ and put $x=\frac{s_1+s_2}{2}$, with $s_i\in O_i$, $i=1,2$.

We can assume that, for $i=1,2$, $O_i=\{x\in B_{C_0(K,X)}:\vert f_{i,j}(x)-\alpha_{i,j}\vert<\varepsilon_{i,j},1\leq j\leq k_i\}$, where $k_i\in \natu$, $f_{i,j}\in S_{C_0(K,X)^*}$ and $\alpha_{i,j}\in [-1,1]$ for $1\leq j\leq k_i$. Define $\mu_{i,j}:=\varepsilon_{i,j}-\vert f_{i,j}(x)-\alpha_{i,j}\vert$ and $\varepsilon:=\min_{1\leq i\leq 2,1\leq j\leq k_i}\mu_{i,j}>0$.

Note now that $B_{C_0(K,X)^*}=\overline{\rm co}\{\delta_t\otimes x^*:t\in K, x^*\in S_{X^*}\}$ ($\delta_t$ denotes the evaluation functional at $t$ as an element in $C_0(K)^*$). Indeed, as $C_0(K,X)^*=C_0(K)^*\hat{\otimes}_{\pi} X^*$ and $B_{C_0(K,X)^*}=\overline{\rm co}(S_{C_0(K)^*}\otimes S_{X^*})$, we get that $B_{C_0(K,X)^*}=\overline{\rm co}\{\delta_t\otimes x^*:t\in K, x^*\in S_{X^*}\}$, since $K$ is scattered and so $C_0(K)^*=\ell_1(K)$. Then, for every $i\in\{1,2\}$ and every $1\leq j\leq k_i$ we can find out $N_{ij}\in\mathbb{N}$ and finite subsets $E\subset K$ and $A\subset S_{X^*}$, both with $N_{ij}$ elements, so that $\Vert g_{ij}-f_{ij}\Vert<\varepsilon/3$, where $g_{ij}\in {\rm co}(\delta_{E}\otimes A)$ being $\delta_{E}=\{\delta_t:t\in E\}$.

Now we get that $$U_i:=\{y\in B_{C_0(K,X)}:\vert g_{ij}(y-s_i)\vert <\varepsilon/3, 1\leq j\leq k_i\}\subset O_i\ \forall i\in\{1,2\}.$$
For every $t\in E$ and for every $i\in\{1,2\}$ we define
$$B_i^t=\{z\in B_X:\vert x^*(z-s_i(t))\vert <\ep/4\ \forall x^*\in A\}.$$
It is clear that $B_i^t$ is a nonempty, convex and relatively weak open subset of $B_X$. Furthermore, as $X$ is weak stable, we have that $O_t:=\frac{B_1^t+B_2^t}{2}$ is a relatively weak open subset of $B_X$ containing $x(t)$, for every $t\in E$. So there are $\delta>0$ and $x_1^*,\ldots ,x_m^*\in B_{X^*}$ such that $$\{z\in B_X:\vert x_i^*(z-x(t))\vert<2\delta,\ 1\leq i\leq m\}\}\subset O_t.$$
Put $$U=\{y\in B_{C_0(K,X)}:\vert x^*(y(t)-x(t))\vert<\delta\ \forall t\in E,\ \forall x^*\in A\cup\{x_1^*,\dots,x_m^*\}\}.$$
It is clear that $U$ is a relatively weak open subset of $B_{C_0(K,X)}$ containing $x$. Now our goal will be to show that $U\subset\frac{O_1+O_2}{2}$, which is enough to finish the proof.

Take $y\in U$. Then we may choose a finite family of disjoint open subsets $\{V_t:t\in E\}$ of $K$ such that $t\in V_t$ for every $t\in E$ and satisfying
$$\vert x^*(y(s)-y(t))\vert<\delta\ \forall s\in V_t,\ \forall x^*\in A\cup\{x_1^*,\dots,x_m^*\}$$ for every $t\in E$. Now it is clear that $y(s)\in O_t$ for every $s\in V_t$. 

Our next step will be to define for $i=1,2$ a function $s_i:K\rightarrow X$. If $s\in K\setminus \bigcup\limits_{t\in E}V_t$ we define $\widetilde{s_i}(s)=y(s)$. Given $t\in E$, before define $s_i(s)$ for $s\in V_t$, we consider the multivalued map $F_t: O_t\twoheadrightarrow X^2$ given by 
$$F_t(z)=\left\{(z_1,z_2)\in \overline{B_1^t}\times\overline{B_2^t}:\frac{z_1+z_2}{2}=z\right\}\ \forall z\in O_t.$$
Note that $O_t$ is paracompact in norm, since $X$ is. Also, $X^2$ is a Fr\'echet space and $F_t(z)$ is a nonempty, closed and convex subset for every $z\in O_t$. Then, in order to apply the Michael selection principle, it is enough to check that $F_t$ is a norm lower semicontinuous multivalued map. To do this we will apply Lema \ref{redes}, so  pick $\{z_n\}$ a sequence in $O_t$ converging in the norm topology to some point $z\in O_t$. Also we pick $(p_1,p_2)\in \overline{B_1^t}\times \overline{B_2^t}$ such that $\frac{p_1+p_2}{2}=z$. As $O_t=\frac{B_1^t +B_2^t}{2}$, we can take $(q_1,q_2)\in B_1^t\times B_2^t$ such that $\frac{q_1+q_2}{2}=z$. From Lemma \ref{sequences} there are a strictly increasing map $\sigma:\mathbb{N}\rightarrow\mathbb{N}$ and  sequences
$ \{(p_1^n,p_2^n)\}$, $ \{(q_1^n,q_2^n)\}$ in $B_X\times B_X$ so that for every $i=1,2$ $\{p_i^n\}$ converges to $p_i$, $\{q_i^n\}$ converges to $q_i$ and
$$\frac{p_1^n+p_2^n}{2}=\frac{q_1^n+q_2^n}{2}=z_{\sigma(n)} \ \forall n\in\mathbb{N}.$$
For $i=1,2$, $\lambda\in [0,1]$ and $n\in \mathbb{N}$ define $v_i^n(\lambda)=(1-\lambda)q_i^n+\lambda p_i^n$ so that $\{v_i^n\}$ converges  $v_i(\lambda):=(1-\lambda)q_i+\lambda p_i$. As $p_i\in \overline{B_i^t}$, $q_i\in B_i^t$ and $B_i^t$ is convex we get that $v_i(\lambda)\in B_i^t$ for every $\lambda\in [0,1)$. Take now a sequence $(t_r)_{r\in\N}\subset [0,1)$ converging to $1$. As $B_i^t$ is open, we can choose a strictly increasing sequence $\{N_r\}\in\N$ such that $v_i^{N_r}(t_r)\in B_i^t$. Finally we put $y_r:=(v_1^{N_r}(t_r),v_2^{N_r}(t_r))$. It is clear that $y_r\in F_t(z_{\sigma(N_r)})$ for every $r$, and $\{y_r\}$ converges to $(p_1,p_2)$, which proves that $F_t$ is norm l.s.c.

Appliyng the Michael selection principle to $F_t$, we can take a continuous selection $f_t=(f_1^t,f_2^t)$ of $F$. From Urysohn lemma there is a continuous function  $n_t:K\rightarrow[0,1]$ such that $n_t(t)=1$ and $n_t(s)=0$ for every $s\in K\setminus V_t$. Now we define for $i=1,2$
$$\widetilde{s_i}(s)=n_t(s)f_i^t(y(s))+(1-n_t(s))y(s)\ \forall s\in V_t.$$
Recall that we had defined before $\widetilde{s_i}(s)=y(s)$ for every $s\in K\setminus \bigcup\limits_{t\in E}V_t$. It is clear that $\widetilde{s_i}\in C_0(K,X)$ for $i=1,2$ and $y=\frac{\widetilde{s_1}+\widetilde{s_2}}{2}$.
 \end{proof}

It is worth to mention that this result generalizes a known result  \cite[Th. 2.5 b)]{abralim2} for finite dimensional $X$. In fact, in \cite{abralim2} it is considered the property (co):

\begin{definition}
A Banach space $X$ is said to be (co) if for every $x,s_1,s_2\in B_X$, $\lambda\in(0,1/2]$ satisfying that $x=(1-\lambda)s_1+\lambda s_2$ and  $\ep>0$, there is $\delta>0$ such that there exist two continuous functions
\begin{equation}\label{co}v_i:B_X(x,\delta)\cap B_X\rightarrow B_X(s_i,\ep)\cap B_X\;\;\;\;\forall i\in\{1,2\}\end{equation}
satisfying $y=(1-\lambda)v_1(y)+\lambda v_2(y)$ for every $y\in B(x,\delta)\cap B_X$.
\end{definition}

Using our selection technique, it is easy to prove that property (co) is equivalent to norm stability of the unit ball for arbitrary Banach spaces.

\begin{proposition}\label{co}
Let $X$ be a Banach space. $X$ is (co) if and only if $B_X$ is stable.
\begin{proof}
Property (co) implies the stability of $B_X$ by Proposition 2.2 of \cite{abralim2}. Now, if $B_X$ is stable and we have $x,s_1,s_2\in B_X$, $\lambda\in(0,1/2]$ and $\ep>0$ such that $x=(1-\lambda)s_1+\lambda s_2$, then we consider the multivalued map given by
$$\phi:B_X(x,\delta)\cap B_X\rightarrow \big(B_X(s_1,\ep)\cap B_X\big)\times \big(B_X(s_2,\ep)\cap B_X\big)$$
$$\phi(y)=\big\{(z_1,z_2)\in B_X\times B_X\;:\;(1-\lambda)z_1+\lambda z_2=y ,\;z_i\in \overline{B_X(s_i,\ep/2)},\,i=1,2\big\}$$
where $\delta>0$ is such that $B_X(x,\delta)\cap B_X\subset (1-\lambda)B_X(s_1,\ep/2)\cap B_X+\lambda B_X(s_2,\ep/2)\cap B_X$ (which exists because $B_X$ is stable). By the same argument as in Theorem \ref{vectorvalued} we are able to prove that $\phi$ is norm l.s.c. so applying the Michael Selection Principle to $\phi$ we are done.
\end{proof}
\end{proposition}

Proposition \ref{co} turns Theorem 2.5 (b) of \cite{abralim2} into a complete characterization since weak and norm stability agree on finite dimensional spaces. Observe that the weak stability of the unit ball in $C_0(K,X)$ implies the weak stability of the unit ball in both $C_0(K)$ and $X$ \cite{abralim2}, since $C_0(K)$ and $X$ are $1$-complemented in $C_0(K,X)$. We don't know if the converse is true.

\section{Relationship with diameter two properties}

As Theorem \ref{predual} is a geometric characterization of purely atomic preduals of $L_1(\mu)$ spaces, one can get, as a consequence, geometric implications of such spaces in relation with other well known geometric properties, as diameter two properties (see \cite{jfa} and references there for background). In \cite{Hal+19} it is defined a very strong diameter 2 property called the symmetric strong diameter 2 property (SSD2P for short).

\begin{definition}
Let $X$ be a Banach space. We will say that $X$ has the  symmetric strong diameter 2 property (SSD2P for short) whenever for every finite family $\{S_i\}_{i=1}^n$ of slices of $B_X$ and $\ep>0$, there exist $x_i\in S_i$ and $y\in B_X$, independent of $i$, such that $x_i\pm y\in S_i$ for every $i\in\{1,\dots,n\}$ and $||y||>1-\ep$.
\end{definition}

In view of this property, we define the next property, which will be proved to be satisfied by Banach spaces with stable unit balls.

\begin{definition}
Let $X$ be a Banach space. We will say that $X$ has the  attaining symmetric strong diameter 2 property (ASSD2P for short) whenever for every finite family $\{S_i\}_{i=1}^n$ of slices of $B_X$, there exist $x_i\in S_i$ and $y\in S_X$, independent of $i$, such that $x_i\pm y\in S_i$ for every $i\in\{1,\dots,n\}$.
\end{definition}

Obviously, the ASSD2P implies the SSD2P which in particular implies the strong diameter 2 property (see \cite{Hal+19}). In order to obtain a characterization of the ASSD2P, we need the concept of symmetric convex combination of slices. Given $S_i$ slices for $i\in\{1,\dots,n\}$, it is defined the symmetric convex combination of those slices as a set of the form
$$\frac{\sum\limits_{i=1}^n\lambda_i S_i-\sum\limits_{i=1}^n\lambda_i S_i}{2}$$
where $\lambda_i>0$ with $\sum\limits_{i=1}^n\lambda_i=1$.
\begin{proposition}\label{ASSD2P}
Let $X$ be an infinite-dimensional Banach space. Then,
\begin{enumerate}
\item[i)] $X$ has the ASSD2P if, and only if, every intersection of symmetric convex combinations of slices reaches the sphere.
\item[ii)] $X$ has the SSD2P if, and only if, for every $\ep>0$ and every intersection of symmetric convex combinations of slices $C$ there exists $y\in C$ of norm $||y||>1-\ep$.
\end{enumerate}
\begin{proof}
i) Let $\{S_i\}_{i=1}^n$ be a finite family of Slices of the unit ball and consider the set
$$C=\bigcap\limits_{i=1}^n\frac{S_i-S_i}{2}.$$
If we assume that intersection of symmetric convex combinations of slices reaches the sphere then there is $y\in S_X\cap C$. Then, for every $i\in\{1,\dots,n\}$ there are two points $s_i,z_i\in S_i$ such that $y=\frac{s_i-z_i}{2}$. Considering $x_i=\frac{s_i+z_i}{2}$ we are done since $x_i-y=z_i\in S_i$ and $x_i+y=s_i\in S_i$. To prove the other implication, consider for every $i\in\{1,\dots,n\}$ a symmetric convex combination of slices $\frac{C_i-C_i}{2}$ where $C_i$ is an arbitrary convex combination of slices in $X$. By Proposition 2.1 of \cite{Hal+19} $a)\Leftrightarrow c)$, we may take for every $i\in\{1,\dots,n\}$ a point $x_i\in C_i$ and a point $y\in S_X$ such that $x_i\pm y\in C_i$ for every $i\in\{1,\dots,n\}$. We are now done since
$$y=\frac{(x_i+y)-(x_i-y)}{2}\in\frac{C_i-C_i}{2}\;\;\;\;\forall i\in\{1,\dots,n\}.$$
The assertion ii) of the Proposition is proved in the same way.
\end{proof}
\end{proposition}

We thank an anonymous referee of an old version of this paper for asking to us about the possibility that weak stability of the unit ball for $L_1$ preduals implies ASSD2P. The next corollary answers positively this question and even more.

\begin{corollary} Banach spaces with a weak stable unit ball satisfy ASSD2P. \end{corollary}

The question now is wether the reverse result is true or not. For that purpose, we are going to give an example of a space with the ASSD2P whose unit ball is not weak stable. $C[0,1]$ does not have a weak stable unit ball as it is a non purely atomic $L_1$-predual. Let us prove that $C[0,1]$ is the example that we are looking for.
\begin{proposition}\label{notASS}
$C[0,1]$ enjoy the ASSD2P.
\begin{proof}
Let $\ep\in(0,1)$ and $\mu_1,\dots,\mu_n\in S_{\mathcal{M}[0,1]}$, so that we consider the slices $S_i$ given by
$$S_i=\{f\in B_{C[0,1]}\;:\;|\mu_i(f)|>1-2\ep\} \;\;\forall i\in\{1,\dots,n\}$$
and we take for every $i\in\{1,\dots,n\}$ an element $f_i\in S_i$ such that $|\mu_i(f_i)|>1-\ep$.
First, we take $r\in(0,1)$ and $\delta>0$ such that
$$\mu_i(r-\delta,r+\delta)<\frac \ep2 \;\;\;\forall i\in\{1,\dots,n\}.$$ Let $g:[0,1]\to[0,1]$ be a continuous function such that $g(t)=1$ for every $t\in[0,1]\setminus(r-\delta,r+\delta)$ and $g(t)=0$ for every $t\in[r-\delta/2,r+\delta/2]$. If we define $g_i=f_ig$ it is straightforward that $g_i\in B_{C[0,1]}$ is such that
$$\begin{aligned}|\mu_i(g_i-f_i)|&=\bigg| \int_{[0,1]}g_i-f_i\,d\mu_i \bigg|\\&\le\int_{(r-\delta,r+\delta)}|g_i-f_i|\,d\mu_i+\int_{[0,1]\setminus(r-\delta,r+\delta)}|g_i-f_i|\,d\mu_i<\ep,\end{aligned}$$
so $g_i\in S_i$. Now, it is enough to consider a continuous function $h:[0,1]\to[0,1]$ such that $h(t)=0$ for $t\in[0,1]\setminus(r-\delta/2,r+\delta/2)$ and $g(r)=1$. It is easy to check that $g_i+h,g_i-h\in S_i$ as $g_i$ and $h$ have disjoint support, where $h\in S_{C[0,1]}$.
\end{proof}
\end{proposition}
Thanks to Proposition \ref{notASS} we have  that
$$w\text{-stable }\nLeftarrow\text{ ASSD2P}.$$

Now, let us prove that the ASSD2P is strictly stronger than the SSD2P.\\

From \cite{HWW93} page 168 we know there exists a Banach space $X$ which is both Almost Square (ASQ) and strictly convex. This means that $X$ has the Symmetric Strong Diameter 2 Property (SSD2P) as it is an Almost Square Banach space \cite{Hal+19}. Since the ASSD2P implies the Attaining Strong Diameter 2 Property (see Theorem 3.4 of \cite{lo}), which is incompatible with strictly convex spaces (Proposition 2.3 of \cite{lo}). It follows that $X$ does not have the ASSD2P, meaning that
$$\text{ASSD2P }\nLeftarrow\text{ SSD2P.}$$
As purely atomic $L_1$ preduals satisfy ASSD2P, it is natural wondering what is the class of $L_1$ preduals with ASSD2P.

We have shown that stability arises a lot of geometric consequences related with diameter 2 properties. Finally, let us now stablish the next Proposition \ref{extremes} focusing on some geometric consequences relating extreme points.

\begin{proposition}\label{extremes}
Let $X$ be a Banach space with a weak stable unit ball. Then:
\begin{enumerate}
\item[i)] Every face of $B_X$ has a weakly closed set of extreme points and so $Ext(B_X)$ is weakly closed.
\item[ii)] $Ext(B_X)$ is nowhere weakly dense in $B_X$, that is, for every $U\subset B_X$ nonempty relatively weakly open subset, there exists another $V\subset U$ nonempty relatively weakly open subset such that $V\cap Ext(B_X)=\emptyset$.
\end{enumerate}
\begin{proof}
From \cite[Proposition 1.1]{Papa}, we know that the multivalued map $T:C\rightarrow C$ which gives for every $x\in C$ the face generated by $x$, that is 
$$T(x)=\{y\in C\;:\; \exists z\in C,\;\exists \lambda\in(0,1)\;\;\text{with}\;\;x=\lambda y+(1-\lambda)z\}$$
is l.s.c. when $C$ is a stable convex set. In particular, it says that $Ext(C)$ is closed when $C$ is stable.

The first assertion is due to that and the fact that every face of a stable convex set is again stable. In fact, if $C$ is a stable convex set and $F$ is one of its faces, then, taking $U_1,U_2\subset C$ relatively open sets and making use of the definition of face, we have that
$$\frac{U_1\cap F+U_2\cap F}{2}=\left(\frac{U_1+U_2}{2}\right)\cap F,$$
so $F$ is stable.

For the second assertion, let us consider $U$ a nonempty relatively weakly open subset of $B_X$. Then by Bourgain's Lemma \cite[Lemma II.1]{ggms} we know there is a nontrivial convex combination of disjoint slices $V$ inside $U$. As slices are relatively weakly open subsets of $B_X$ and $B_X$ is weak stable, we know that $V$ is a nonempty relatively weakly open. It is straightforward to see that $V$ cannot have any extreme point of $B_X$, as it is a convex combination of more than one nonempty disjoint subsets of $B_X$.
\end{proof}
\end{proposition}


{\it \bf Acknowledgements.-} We want to thank Miguel Mart\'{\i}n and Abraham Rueda for their help looking for some adequate reference.

\end{document}